\documentclass[11pt]{amsart}

\usepackage{amsmath,graphics}
\usepackage{amsfonts,amssymb}
\usepackage{amscd}
\usepackage{mathrsfs}
\usepackage[all]{xy}
\usepackage{accents}
\usepackage{array}
\usepackage{tikz}
\usepackage{fancyvrb}
\usepackage{enumitem}

\theoremstyle{plain}
\newtheorem*{theorem*}{Theorem}
\newtheorem*{lemma*} {Lemma}
\newtheorem*{corollary*} {Corollary}
\newtheorem*{proposition*} {Proposition}
\newtheorem{theorem}{Theorem}[section]
\newtheorem{lemma}[theorem]{Lemma}
\newtheorem{corollary}[theorem]{Corollary}
\newtheorem{proposition}[theorem]{Proposition}

\newtheorem{prelemma*}{Prelemma}

\theoremstyle{remark}

\newtheorem{notation}{Notation}
\newtheorem{example}{Example}

\theoremstyle{definition}

\newtheorem{conj}{Conjecture}

\def \Z {\mathbb{Z}}
\def \C {\mathbb{C}}
\def \P {\mathbb{P}}

\def \Qt{\widetilde{Q}}
\def \Qtl{\widetilde{Q}_\lambda}
\def \Pt{\widetilde{P}}
\def \Ptl{\widetilde{P}_{\lambda}}

\def \Rn{R(n)}
\def \Dn{\mathcal{D}(n)}

\def \bn{\begin{enumerate}}
\def \en{\end{enumerate}}
\def \bdm{\begin{displaymath}}
\def \edm{\end{displaymath}}

\def \bp{\begin{proof}}
\def\ep{\end{proof}}

\def\U+e{(U^+)^e}

\def\trm{\hspace{0.08in}\textrm{for} \hspace{0.08in}}
\def\be{\begin{equation}}
\def\ee{\end{equation}}

\def\Fdot{ F_{\mbox{\boldmath{.}}}}

\def\og{OG(n)}
\def\lg{LG(n)}
\def\In{\mathcal{I}_n}

\def\dn{\mathcal{D}(n)}
\def\Rn{\mathcal{R}(n)}

\begin{document}

\title [Galkin's Lower bound Conjecure for $\lg$ and $\og$] {Galkin's Lower bound Conjecure for LAGRANGIAN AND ORTHOGONAL GRASSMANNIANS}
\author{Daewoong Cheon and Manwook Han}
\address{Chungbuk National University, Department of Mathematics,
Chungdae-ro 1, Seowon-Gu, Cheongju City, Chungbuk 28644, Korea}
\email{daewoongc@chungbuk.ac.kr}
\address{Chungbuk National University, Department of Mathematics,
Chungdae-ro 1, Seowon-Gu, Cheongju City, Chungbuk 28644, Korea}
\email{santamaria95@chungbuk.ac.kr}
\date{\today}

\begin{abstract}
Let $M$ be a Fano manifold, and $H^\star(M;\C)$ be the quantum cohomology ring of $M$ with the quantum product $\star.$ For  $\sigma \in H^*(M;\C)$, denote by $[\sigma]$  the quantum multiplication operator $\sigma\star$ on $H^*(M;\C)$. It was conjectured several years ago \cite{GGI, GI} 
 and has been proved for many  Fano manifols \cite{CL1, CH2, LiMiSh, Ke}, including our cases, that  the operator $[c_1(M)]$ has a real valued eigenvalue $\delta_0$ which is  maximal among eigenvaules of $[c_1(M)]$.
Galkin's lower bound conjecture \cite{Ga} states that for a Fano manifold $M,$ $\delta_0\geq \mathrm{dim} \ M +1,$ and the equlity holds if and only if $M$ is the projective space $\mathbb{P}^n.$ In this note, we show that Galkin's lower bound conjecture holds for Lagrangian and orthogonal Grassmannians, modulo some exceptions for the equality.
\end{abstract}
\maketitle

\section{INTRODUCTION}
Let $M$ be a Fano manifold. The quantum cohomology ring $H^\star(M,\C)$  is a complex vector space $H^*(M,C)$ together with the quantum multiplication $\star$, a deformation of the cup product of the ring $H^*(M,C)$.  For $\sigma\in H^\star(X,\C),$ define the quantum multiplication operator $[\sigma]$ on $H^\star(X,\C)$ by $[\sigma](\tau)=\sigma\star \tau$ for $\tau \in H^\star(X,\C)$. Galkin's lower bound conjecture (for short, GLBC) is related with and preceded by the so-called Property $\mathcal{O}$ which some manifolds may enjoy. So let us first address Property $\mathcal{O}$. See \cite{GGI, GI} for details on property $\mathcal{O}$. Let $\delta_0=\delta_0(M)$ be the absolute value of a maximal  eigenvalue of the operator $[c_1(M)]$, where $c_1(M)$ denotes the first Chern class of the tangent bundle of $M.$
A manifold  $M$ is said to have property $\mathcal{O}$ if  \bn
\item $\delta_0$ is an eigenvalue of $[c_1(M)]$.
\item The multiplicity of the eigenvalue $\delta_0$ is one.
\item If $\delta$ is an eigenvalue of  $[c_1(TM)]$ such that $|\delta|=\delta_0,$ then $\delta =\delta_0\xi$ for some $r$-th root of unity, where $r$ is the Fano index of $M.$
\en

GLBC concerns the real number $\delta_0$. Once a manifold $M$ has property $\mathcal{O},$ then $\delta_0$ would become an eigenvalue of $[c_1(M)]$. In fact, Galkin, Golyshev and Iritani conjectured that Fano manifolds have property $\mathcal{O}$ (\cite{GGI, GI}), and it turned out that this is the case for many  Fano manifolds  \cite{CL1, CH2, LiMiSh, Ke}.
In particular, we have

\begin{proposition}[\cite{CL1}] \label{homogeneous-O} General homogeneous manifolds $G/P$  have property $\mathcal{O}.$ 
\end{proposition}

 Let $\lg$ (resp. $\og$) denote Lagrangian (resp. orthogonal) Grassmannian (see below for the precise definitions). Then by  Proposition \ref{homogeneous-O}, the $\delta_0$ is an eigenvalue of $[c_1(M)]$ for $\lg$ and $ \og$. A direct proof of this also was given in  \cite{CH2}.  

\smallskip
Now we  state Galkin's lower bound conjecture. 
\begin{conj}[Galkin's lower bound conjecture, \cite{Ga}] For a Fano manifold $M,$ we have $$\delta_0\geq \mathrm{dim} \ M +1,$$ and the equlity holds if and only if $M$ is the projective space $\mathbb{P}^n.$
\end{conj}

Very recently,  Galkin's lower bound conjecure was proved for the Grassmannian \cite{EvScShShWa}.
Motivated by this, we have investigated if  GLBC holds for $\lg$ and $\og,$ and have obtained the following.

\begin{theorem}\label{Main}
GLBC holds for $\lg$ and $\og$. To be precise, we have 
\bn
\item $\delta_0(\lg)\geq \mathrm{dim} \ \lg$,
\item $\delta_0(\og)\geq \mathrm{dim} \ \og$.
\en
 Furthermore,  the equality  holds  for $LG(1)$, $OG(1)$ and $OG(2).$ 
\end{theorem}

Note that $LG(1)$ is none other than the projecive line $\mathbb{P}^1$, for which the equality in GLBC was verified in \cite{EvScShShWa}.  However, $OG(1)$ and $OG(2)$ are new examples for which the equality holds. 

  To prove the conjecture, as in \cite{EvScShShWa}, we use the explicit computation of $\delta_0$ made in \cite{CH2} together with some calculus.  But since the  behaviour of $\delta_0$ (depending on $n$) is quite different from  that of the Grassmannian,
we take a diffenent calculus apprach to show Theorem \ref{Main}. At the last part of the article we also make a dirct check on the equlity in GLBC for $OG(1)$ and $OG(2)$, using the quantum Pieri rule for $\og$.

\subsection*{Acknowledgements} The first author was supported by Basic Science Research Programs through the National Research Foundation of Korea (NRF) funded by the Ministry of Education (NRF-2018R1D1A3B07045594).
\section{Preliminaries}
In this section, we review $\Qt$- and $\Pt$-polynomials of Pragacz
and Ratajski which represent cohomology classes of Lagrangian and orthogonal Grassmannians, respectively. The readers refer  to  \cite{Pra and Rat 1} and \cite{Las and Pra} for details. Let us begin with some combinatorial notations.
\subsection{Notations} Fix a positive integer $n$. A \textit{partition} $\lambda$ is a weakly decreasing sequence of nonnegative integers $\lambda = (\lambda_1 , \lambda_2 , \ldots , \lambda_n)$. The nonzero $\lambda_i$ are called the \textit{parts} of $\lambda$. The number of parts is called the \textit{length} of $\lambda$ and is denoted $l(\lambda)$. The sum $\sum_{i=1}^n \lambda_i$ is called the \textit{weight} of $\lambda$, and denoted $|\lambda|$. Let $\Rn$ be the set of all partitions $(\lambda_1 , \ldots , \lambda_n)$ such that $\lambda_1 \leq n$. 
A partition $\lambda$ of length $l$ is called \textit{strict} if $\lambda_1 > \cdots >\lambda_l$. Let $\Dn\subset \Rn$ be the set of all strict partitions $(\lambda_1, \ldots ,\lambda_n)$ such that $\lambda_1 \leq n$. We usually write $(\lambda_1, \ldots ,\lambda_l)$ for a (strict) partition $\lambda = (\lambda_1, \ldots , \lambda_l, 0, \ldots , 0)$ of length $l$, if no confusion should arise.
 For $\lambda \in \mathcal {D}(n)$, 
$\widehat{\lambda}$ denotes the partition whose parts complements the parts
of $\lambda$ in the set $\{1,...,n\}$.
To a partition $\lambda$, we associate a Young diagram $T_\lambda$
of boxes (see \cite{F}). Write $\lambda\supset \mu$ if  $T_\lambda\supset T_\mu$. For $\lambda, \mu$ with
$\lambda \supset \mu$, $\lambda/\mu$ denotes the \textit{skew diagram }
$T_\lambda\setminus T_\mu$ defined by  the set-theoretic difference. A skew diagram $\alpha$ is
a \textit{horizontal strip} if it has at most one box in each column. Two boxes in $\alpha$ are
\textit{connected} if they share a vertex or an edge; this defines the connected components of $\alpha.$

\subsection{$\Qt$- and $\Pt$-polynomials}
Let $X = (x_1, \ldots , x_n)$ be an $n$-tuple of variables. For $i=1, \ldots , n$, let  $E_i (X)$ be the $i$-th  elementary symmetric polynomial in $X$.
The $\Qt$-polynomials of Pragacz and Ratajski are indexed by the elements of $\Rn$. For $i \geq j$, define
\begin{displaymath}
\Qt_{i,j}(X) \ = \ E_i (X) E_j (X) + 2 \sum_{k=1}^{j}(-1)^k E_{i+k} (X) E_{j-k} (X) .
\end{displaymath}
For any partition $\lambda$, not necessarily strict, and for $r = 2 \left\lfloor (l(\lambda)+1)/2 \right\rfloor$, let $B_\lambda$ be the $r \times r$ skewsymmetric matrix whose $(i, j)$-th entry is given by $\Qt_{\lambda_i , \lambda_j} (X)$ for $i < j$. The $\Qt$-polynomial associated to $\lambda$ is defined by
\[ \displaystyle \Qt_{\lambda}(X) \ = \ \textrm{Pfaffian} (B_{\lambda}) . \]
 $\Pt$-polynomials are defined as follows.
Given $\lambda$, not necessarily strict,  define
$$\Ptl(X):=2^{-l(\lambda)}\Qtl(X).$$
Note that from the definitions, for $\lambda = (k)$ with $0 \leq k \leq n$ we have $\Qt_{(k)}(X) = E_k (X)=2 \Pt_{(k)}$. We often write $\Qt_k(X)$ and $\Pt_k(X)$ for $\Qt_{(k)}(X)$ and $\Pt_{(k)}(X)$, respectively.

\section{Quantumtum cohomology}
This section is devoted to giving a desctription of the quantum cohomology rins of Lagrangian and orthogonal Grassmannians. The readers can refer to \cite{KT2, KT1} and references therein for details.
\subsection{Lagrangian and orthogonal Grassmannians} Let $E=\C^{2n}$ be a complex vector space.  Fix a symplectic form, i.e., a nondegenerate skew symmetric bilinear form $Q$ on $E$.  A subspace
$\Sigma \subset E$ is called $isotropic$ if $Q(v,w)=0$ for all $v,w \in
\Sigma.$ By a linear algebra,  a  maximal isotropic subspace of $E$, called a Lagrangian subspace, is has a complex dimension $n$. Let $\lg$ be the parameter space of Lagrangian subspaces in $E$. Then $\lg$ is a homogeneous manifold
$Sp_{2n}(\C)/P_n$ of (complex) dimension $n(n+1)/2$, where $P_n$ is
the maximal parabolic subgroup of the symplectic group $Sp_{2n}(\C)$
associated with the `right end root' in the Dynkin Diagram of Lie
type $C_n$,  e.g., on Page $58$ of \cite{Hu1}.

To define an orthogonal Grassmannian, fix  a symmetric bilinear form $Q$ on $E=\C^N$.  A maximal isotropic subspace of  has complex dimension $\lfloor\frac{N}{2}\rfloor$. For a positive integer $n$, let $OG^o(n)$ (resp. $OG^{e}(n+1)$) be the parameter space of maximal isotropic subspaces in $E=\C^{2n+1}$ (resp. $E=\C^{2n+2}$).  Then the odd orthogonal Grassmannian $OG^o(n)$ is a homogeneous variety $SO_{2n+1}(\C)/P_{n}$ of
dimension $n(n+1)/2$, where $P_{n}$ is the maximal parabolic
subgroup of $SO_{2n+1}(\C)$ associated with a `right end root' of
the Dynkin diagram of type $B_n$.
The variety $OG^e(n+1)$ has two isomorphic components. Fix one component  of $OG^e(n+1)$ for which we write $OG^e(n+1)^\prime$. Then  the {\it even orthogonal Grassmannian} $OG^e(n+1)^\prime$ is
a homogeneous variety  $SO_{2n+2}(\C)/P_{n+1}$, where $P_{n+1}$
is the maximal parabolic subgroup of $SO_{2n+2}(\C)$ associated with
the `right end root' of the Dynkin diagram of Lie type $D_{n+1}$. It is well-known that two varieties $OG^o(n)$ and  $OG^e(n+1)^\prime$ are isomorphic to each other. In what follows, we will identify
these two varieties, write $$OG(n):= OG^o(n)=OG^e(n+1)^\prime.$$

\subsection{Quantum cohomology of $LG(n)$} The (quantum) cohomology ring of a homogeneous manifold $G/P$ can be best described in terms of  the Schubert classes, which are, in turn, given by the Schubert varieties of $G/P.$ For $\lg$, the Schubert varieties are defined as follows.  Let $E=\C^{2n}$ be a vector space with a symplectic form $Q$.
Fix a
complete isotropic flag $ F_{\mbox{\boldmath{.}}}$ of subspaces
$F_i$ of $E$:
$$\Fdot : 0=F_0\subset F_1\subset \cdots \subset F_n\subset E,$$
where dim$(F_i)=i$  for each $i,$ and $F_n$ is Lagrangian. For 
$\lambda \in \Dn,$ the Schubert variety $X_{\lambda}(\Fdot)$ associated with $\lambda$ is defined as 
\be \label{schubert} X_{\lambda}(\Fdot)= \big\{\Sigma \in \lg \ | \ \textrm{dim}(\Sigma \cap
F_{n+1-\lambda_i})\geq i \trm i=1,...,l(\lambda)\big \}.\ee Then
$X_{\lambda}(\Fdot)$ is a subvariety of $\lg$ of complex codimension
$| \lambda |.$ The Schubert class associated with $\lambda$ is
defined to be the cohomology class, denoted by $\sigma_{\lambda}$,
Poincar\'{e} dual to the homology class $[X_{\lambda}(\Fdot)]$, so
 $\sigma_\lambda \in H^{2|\lambda|}(\lg,\Z).$ It is a classical
result that
$\{\sigma_{\lambda}\hspace{0.05in}|\hspace{0.05in}\lambda\in \dn\}$
forms an additive basis for $H^{*}(\lg,\Z)$. It is conventional to
write $\sigma_i$ for $\sigma_{(i)}$. 

A rational map of degree $d$ to $\lg$ is a morphism
$f:\mathbb{P}^1\rightarrow \lg$ such that
$$\int_{LG(n)}f_*[\P^1]\cdot \sigma_1=d.$$ Given an integer $d\geq0$
and partitions $\lambda, \mu,$ $\nu\in \mathcal{D}(n)$, the
Gromov-Witten invariant
$<\sigma_{\lambda},\sigma_{\mu},\sigma_{\nu}>_d$  is defined as the
number of rational maps $f:\mathbb{P}^1\rightarrow \lg$ of degree
$d$ such that $f(0)\in X_{\lambda}(\Fdot),$ $f(1)\in X_{\mu}(
G_{\mbox{\boldmath{.}}}),$ and $f(\infty)\in X_{\nu}(
H_{\mbox{\boldmath{.}}})$, for given isotropic flags $\Fdot$,
$G_{\mbox{\boldmath{.}}},$ and $H_{\mbox{\boldmath{.}}}$ in general
position. We remark that
$<\sigma_{\lambda},\sigma_{\mu},\sigma_{\nu}>_d$ is $0$ unless
$|\lambda|+|\mu|+|\nu|=\mathrm{dim}(LG(n))+(n+1)d$. The quantum
cohomology ring $qH^{*}(\lg,\Z)$ is isomorphic
 to $H^{*}(\lg,\Z)\otimes \Z[q]$
as  $\Z[q]$-modules, where $q$ is a formal variable of degree
$(n+1)$ and called the $quantum$ $variable$. The multiplication in
$qH^{*}(\lg,Z)$ is given by the relation \be~\label{multi}
\sigma_\lambda \star
\sigma_\mu=\sum<\sigma_\lambda,\sigma_\mu,\sigma_{\hat{\nu}}>_d
\sigma_\nu q^d,\ee  where the sum is taken over $d\geq0$ and
partitions $\nu$ with $|\nu|=|\lambda|+|\mu|-(n+1)d.$\\

Now we a give a presentation of the
ring $qH^*(\lg,\Z)$ and the quantum Pieri rule, both due to Kresch and
Tamvakis \cite{KT2}.
\begin {theorem}[\cite{KT2}] \label{quantum coho:lg}
 The quantum cohomology ring
$qH^*(\lg,\Z)$ is presented as a quotient of the polynomial ring
$\Z[\sigma_{1},...,\sigma_n, q]$ by the relations
$$\sigma_i^2+2\sum_{k=1}^{n-i}(-1)^k\sigma_{i+k}\star\sigma_{i-k}=(-1)^{n-i}\sigma_{2i-n-1}q$$
for $1\leq i \leq n.$ 
\end{theorem}

\begin{notation} For partitions $\lambda, \mu\in \Rn$ with $\mu \supset \lambda$, let $N(\lambda,\mu)$ denote the number of connected components of $\mu/\lambda$ not meeting the first column, and let $N^\prime(\lambda, \mu)$ be one less than the number of connected components of $\mu/\lambda.$
\end{notation} 

\begin{proposition} [Quantum Pieri Rule for $\lg$, \cite{KT2}] For any $\lambda \in \mathcal{D}(n)$ and $k>0$, we have
\begin{displaymath}
\sigma_k \star \sigma_\lambda= \sum_{\mu}2^{N(\lambda,\mu)}\sigma_\mu+\sum_{\nu }2^{N^\prime(\nu,\lambda)} \sigma_{\nu}q,
\end{displaymath}
where 
the first sum is over all partitions $\mu\supset \lambda$ with $|\mu|=|\lambda|+k$ such that $\mu/\lambda$ is a
horizontal strip, and the second is over all strict $\nu$ contained in $\lambda$ with $|\nu|=|\lambda|+k-n-1$ such that $\lambda/\nu$ is a horizontal strip.
\end{proposition}

\begin{example} When $n=1$, the Schubert basis elements are $\sigma_0$ and $\sigma_1$.
By the quantum Pieri rule, we have 
\bn
\item $\sigma_1\star \sigma_0= \sigma_1,$
\item $  \sigma_1 \star \sigma_1=q.$
\en
\end{example}

\begin{example}
When $n=2,$ the Schubert basis consists of $\sigma_0, \sigma_1, \sigma_2, \sigma_{(2,1)}$.
The quantum multiplication by $\sigma_1$ is computed as follows.
\bn
\item $\sigma_1\star \sigma_0=\sigma_1,$
\item $\sigma_1\star \sigma_1=2\sigma_2,$
\item $\sigma_1\star \sigma_2=\sigma_{(2,1)}+q,$
\item $\sigma_1\star \sigma_{(2,1)}=\sigma_1 q.$
\en
\end{example}

\subsection{Quantum cohomology  of the orthogonal  Grassmannian}\label{subsec:ogo}
The quantum cohomology ring of $OG(n)$ can be defined in the same way as that  of   $\lg.$  We proceed with the orthogonal Grassmannian $OG(n)=OG^o(n)$.
Let $E$ be a complex vector space of dimension $2n+1$
equipped with a nondegenerate symmetric form. Fix a complete flag: $$\Fdot : 0=F_0\subset F_1\subset \cdots \subset F_n\subset E.$$ For $\lambda \in
\dn$, the Schubert variety $X_{\lambda}(\Fdot)$  is defined by the
same equation $(\ref{schubert})$ as in the above.  The Schubert class
$\tau_\lambda$ is defined as a cohomology class Poincar\'{e} dual to
$[X_{\lambda}(\Fdot)]$. Then $\tau_\lambda \in
H^{2l(\lambda)}(OG(n),\Z)$, and the cohomology classes
$\tau_\lambda$, $\lambda \in \mathcal{D}(n)$, form a $\Z$-basis for
$H^*(\og,\Z).$ 
\smallskip

 For $OG(n)$, the Gromov-Witten invariants are defined similarly.
 Given an integer $d\geq 0$, and $\lambda,\mu, \nu \in \mathcal{D}(n)$, the Gromov-witten invariant
$<\tau_\lambda, \tau_\mu, \tau_\nu>_d$ is defined as  the number of
rational maps $f:\mathbb{P}^1\rightarrow \og $ of degree $d$ such
that $f(0)\in X_{\lambda}(\Fdot),$ $f(1)\in X_{\mu}(
G_{\mbox{\boldmath{.}}}),$ and $f(\infty)\in X_{\nu}(
H_{\mbox{\boldmath{.}}})$, for given isotropic flags $\Fdot$,
$G_{\mbox{\boldmath{.}}},$ and $H_{\mbox{\boldmath{.}}}$ in general
position. Note that $<\tau_\lambda, \tau_\mu, \tau_\nu>_d=0$ unless
$|\lambda|+|\mu|+|\nu|=\textrm{deg}(\og)+2nd$. The quantum
cohomology ring of $\og$ is isomorphic to $H^*(\og,\Z)\otimes \Z[q]$
as  $\Z[q]$-modules. The multiplication in $qH^*(\og,\Z)$ is given
by the relation \be~\label{multi:2} \tau_\lambda \star
\tau_\mu=\sum<\tau_\lambda,\tau_\mu,\tau_{\hat{\nu}}>_d\tau_\nu
q^d,\ee where the sum is taken over $d \geq0$ and partitions $\nu$
with $|\nu|=|\lambda|+|\mu|-2nd.$

 \begin{theorem}[\cite{KT1}]\label{quantum coho:oge}
 The  ring $qH^{*}(\og,\Z)$ is presented as a quotient
of the polynomial ring $\Z[\tau_1,...,\tau_n,q]$ modulo the
relations $\tau_{i,i}=0$ for $i=1,...,n-1$ together with the quantum
relation $\tau_n^2=q$, where \be \label{tau_ii}
\tau_{i,i}:=\tau_i^2+2\sum_{k=1}^{i-1} (-1)^k \tau_{i+k}\star \tau_{i-k}
+ (-1)^i \tau_{2i}.\ee 
\end{theorem}

\begin{proposition} [Quantum Pieri Rule for $\og$, \cite{KT1}]. For any $\lambda \in \mathcal{D}(n)$ and $k>0$, we have
\begin{displaymath}
\tau_k \star \tau_\lambda= \sum_{\mu}2^{N^\prime(\lambda,\mu)}\tau_\mu+\sum_{\mu\supset (n,n)}2^{N^\prime(\lambda,\mu)} \tau_{\mu\setminus (n,n)}
\end{displaymath}
where both sums are over $\mu  \supset \lambda$ with $|\mu|=|\lambda|+k$ such that $\mu/ \lambda$ is a horizontal
strip, and the second sum is restricted to those $\mu$ with two parts equal to $n.$
\end{proposition}

For a later use, we compute the followings.

\begin{example} \label{Example3}
When $n=1$, the Schubert basis elements are $\tau_0$ and $\tau_1$.
By the quantum Pieri rule, we have 
\bn
\item $\tau_1\star \tau_0= \tau_1,$
\item $  \tau_1 \star \tau_1=q.$
\en
\end{example}

\begin{example}\label{Example4}
When $n=2,$ the Schubert basis consists of $\tau_0, \tau_1, \tau_2, \tau_{(2,1)}$.
The quantum multiplication by $\tau_1$ is computed as follows.
\bn
\item $\tau_1\star \tau_0=\tau_1,$
\item $\tau_1\star \tau_1=\tau_2,$
\item $\tau_1\star \tau_2=\tau_{(2,1)},$
\item $\tau_1\star \tau_{(2,1)}=q.$
\en
\end{example}

\section{Quantum multiplication operators}

In this section, we give a description of eigenvalues of the quantum multiplication operators on $H^\star(M,\C)$. See \cite{CH2} for more details.  

\subsection{Notations}
For $n = 2m+1$, let
\[
\mathcal{T}_n \ := \ \{ J = (j_1, \ldots , j_n) \in \mathbb{Z}^n \: | \: -m \leq j_1 < \dots < j_n \leq 3m+1 \} ,
\]
and for $n = 2m$, let
\[
\mathcal{T}_n \ := \ \left\{ J = (j_1 , \ldots , j_n) \in \left( \mathbb{Z} + \frac{1}{2} \right)^n \: {\big|} \: -m +\frac{1}{2} \leq j_1 < \dots < j_n \leq 3m- \frac{1}{2} \right\} .
\]
For $J = (j_1, \ldots , j_n) \in \mathcal{T}_n$ and $\zeta := \exp \left({\frac{\pi \sqrt{-1}}{n}}\right)$, we write $$\zeta^J := (\zeta^{j_1}, \ldots ,\zeta^{j_n}).$$ Let  $\In$ be a subset of $\mathcal{T}_n$ defined as
\[ \In \ := \ \left\{ J = (j_1, \ldots , j_n) \in \mathcal{T}_n \hspace{0.05in}| \: \zeta^{j_k} \ne -\zeta^{j_l}  \text{ for } k \neq l \right\} . \]
Note that $\prod_k \zeta^{j_k} = \pm 1$ for $J = (j_1 , \ldots , j_n) \in \In$. We put
$$ \In^e \ := \ \left\{ J \in \In \hspace{0.05in}\big|\hspace{0.05in} {\prod}_k \zeta^{j_k} = 1 \right\} . $$

We can easily check that $|\mathcal{I}_n|=2^n=|\mathcal{D}(n)|,$ and $|\In^e|=2^{n-1}$.
\bigskip
\subsection{Eigenvalues of $c_1(M)$}
For $M=\lg$, or $OG(n),$ let $H^\star(M;\C)$ be the specialization of the (complexified) quantum ring $qH^*(M,\C)$ at $q=1$, i.e.,
 $$H^\star(M;\C):=qH^*(M,\C)/<q-1>.$$
Then $H^\star(\lg;\C)$ (resp.  $H^\star(OG(n);\C)$)  is a complex vector space of  dimension $2^n$  with a Schubert basis $\{\sigma_\lambda \ | \ \lambda \in \mathcal{D}(n) \}$ (resp. $\{\tau_\lambda \ | \ \lambda \in \mathcal{D}(n)\} $).

For each $I\in \mathcal{I}_{n+1}^e,$ let $$\sigma_I=\sum_{\nu \in
\mathcal{D}(n)}\Qt_\nu(\delta\zeta^I))\sigma_{\hat{\nu}},$$ where $\delta=(\frac{1}{2})^{\frac{1}{n+1}}$,

\begin{theorem}\label{Main-Thm2}

 For  $\lambda \in \mathcal{D}(n)$,
 the  operator $[\sigma_\lambda]$ on
 $H^\star(\lg,\C)$ has  eigenvectors $\sigma_I$ with eigenvalues
 $\Qt_\lambda(\delta\zeta^I).$ In fact, 
  $\{\sigma_I \hspace{0.04in}|\hspace{0.04in}I \in \mathcal{I}^e_{n+1}\}$
 forms a simultaneous eigenbasis of the vector space $H^\star(\lg,\C)$
for the operators $[F]$ with  $H^\star(\lg,\C)$. 
\end{theorem}

For each $I\in \In,$ let $$\tau_I=\sum_{\nu \in
\mathcal{D}(n)}\Pt_\nu(\epsilon\zeta^I))\tau_{\hat{\nu}},$$ where
$\epsilon=(2)^{\frac{1}{n}}$.
\begin{theorem}\label{Main-Thm1}
 For $\lambda \in \mathcal{D}(n)$,
 the  operator $[\tau_\lambda]$ on  $H^\star(OG(n),\C)$ has
 eigenvectors $\tau_I$ with eigenvalues $\Pt_\lambda(\epsilon\zeta^I).$ In particular,
 $\{\tau_I\hspace{0.03in}|\hspace{0.03in}I \in \mathcal{I}_n\}$
 forms
a simultaneous eigenbasis of the vector space $H^\star(OG(n),\C)$
for the operators $[F]$ with  $F \in H^\star(\og,\C)$.
\end{theorem}

The following result is due to Rietsch.
\begin{lemma}[\cite{Riet1}, p. 542]\label{evaluation1} The evaluation
$E_1(\zeta^{I_0})$ is a positive real number which is equal to
$E_1(\zeta^{I_0})=\frac{1}{\sin(\pi/2n)}.$
\end{lemma}

The first Chern classes of $\og$ and $\lg$ , respectively, are given as follows (see \cite{FW}, Lemma $3.5$):
$$ c_1(\og)=2n\tau_1\hspace{0.2in}\textrm{and}\hspace{0.2in}
c_1(\lg)=(n+1)\sigma_1.$$

By this fact, together with Theorems \ref{Main-Thm2}, \ref{Main-Thm1} and  Lemma \ref{evaluation1}, we obtain

\begin{corollary}\label{Coro:eigenvalue}
 The maximal modulus eigenvalues $\delta_0=\delta_0(M)$ for $M=\lg, \og$ are given as follows.
\bn
\item For $\lg,$ $$\delta_0=2^{-\frac{1}{n+1}}\ (n+1)\ \big(\sin\frac{\pi}{2(n+1)}\big)^{-1}.$$
\item For $\og$, $$ \delta_0=2^{\frac{1}{n}}\ n \ \big(\sin{\frac{\pi}{2n}}\big)^{-1}.$$
 \en 
\end{corollary}

\section{Proof of the main results}
Recall that $$\mathrm{dim}\ \lg=\mathrm{dim}\ \og=\frac{n(n+1)}{2}.$$
To ease notation, write $$d(n)=\frac{n(n+1)}{2}+1.$$
Proving Galkin's lower bound conjecture for $\lg$ and $\og$ is reduced to showing that for each of two $\delta_0$s,  holds the inequality
$$\delta_0\geq d(n).$$
To show these inequalites, we use a ``polynomial approximation to the inequalities".
\begin{lemma}\label{Lemma:LG}Let $f(x)=2x-2^x(2x^2-x+1)\sin\frac{\pi x}{2}$. Then we have
$f(x)\geq 0$ for all $x\in (0,\frac{1}{3}].$
\end{lemma}

\begin{proof}
Note that if $x\in (0,\frac{1}{3}],$ then $\frac{\pi}{2}2^x\leq \frac{\pi}{2}2^{\frac{1}{3}}\leq 2.$ Thus  we have $$f(x)\geq 2x-2^x(2x^2-x+1)\frac{\pi x}{2}\geq 2x-2x(2x^2-x+1):=u(x) ,$$ where we used the inequality $\sin y \leq y$ for all $y.$
We can check that $u^\prime(x)\geq 0$ and hence $u(x)$ is increasing on $(0,\frac{1}{3}]$.  Since $u(0)=0$, we have $u(x)\geq 0$ and hence $h(x)\geq 0$ on $(0,\frac{1}{3}].$
\end{proof}

\begin{lemma}\label{Lemma:OG}Let $$h(x)=2^{x+1}x-(2x^2+x+1)\mathrm{\sin} \frac{ x}{2}\pi.$$
Then we have $h(x)\geq 0$ for all $0\leq x \leq \frac{1}{6}$.
\end{lemma}
\begin{proof} Note that $\mathrm{\sin} \frac{ x}{2}\pi\leq \frac{\pi x}{2}$ and $\frac{\pi}{2}\leq \frac{18}{11}$, $h(x)\geq 2x- \frac{18}{11}(2x^2+x+1)x =:v(x)$.
But, as in the above, we can easily check:  $v(x)\geq 0$ and hence $h(x)\geq 0$  for all $0\leq x \leq \frac{1}{6}$.
\end{proof}

\begin{proof}

\end{proof}

\begin{theorem}
Galkin's lower bound conjecture holds, i.e., $\delta_0\geq d(n)$, for $\lg$ and $\og$, and the equality $\delta_0=d(n)$ holds  for $LG(1)$, $OG(1)$ and $OG(2).$ 
\end{theorem}

\begin{proof}
Case of $\lg:$\\
For $n\geq 2$, the inequality $\delta_0\geq d(n)$ follows from the inequality in Lemma \ref{Lemma:LG} with $x=\frac{1}{n+1}$;
\begin{displaymath} f\big(\frac{1}{n+1}\big)\geq 0.\end{displaymath}
For $n=1,$ we check that the equality, not a strict inquality, holds; $$\delta_0=2=d(1).$$
Case of $\og:$\\
If $n\geq 6$, the inquality follows from the inequality in Lemma $\ref{Lemma:OG}$ with $x=\frac{1}{n}$; $$h\big(\frac{1}{n}\big)\geq 0.$$ The remaining cases can be checked by hands or a calculator as follows.
\begin{enumerate}
\item if $n=1,$ then $\delta_0=2=d(1)$,
\item if $n=2,$ then $\delta_0=4=d(2)$,
\item if $n=3,$ then $\delta_0\approx7.55\geq d(3)=7$,
\item if $n=4,$ then $\delta_0\approx 12.43\geq d(4)=11$,
\item if $n=5,$ then $\delta_0\approx26.02\geq d(5)=22.$
\end{enumerate}
\end{proof}

\subsection{For the equality in GLBC}
Note that  $OG(1)$ and $OG(2)$ are new examples for which the equality holds.
This can also be directly checked as follows.\\
For $M=OG(1)$,  fix the ordered basis $\mathcal{B}_1=\{\tau_0, \tau_1\}$ for $H^\star(M,\C)$. Then by Example \ref{Example3}, the matrix of $[c_1(M)]$ with respect $\mathcal{B}_1$ is given by
\be 
A_1=\left(\begin{array}{cc}
0&2\\
2&0
\end{array}\right).
\ee
Then the eigenvalues of $A_1$ are $2, -2,$ and hence  we have $$\delta_0=2= \mathrm{dim} \ OG(1)+1.$$

For $M=OG(2),$ using Example \ref{Example4}, the matrix of $[c_1(M)]$ with respect to the ordered basis $\mathcal{B}_2=\{\tau_0,\tau_1, \tau_2, \tau_{(2,1)}\}$ is given by
\be 
A_2=\left(\begin{array}{cccc}
0&0&0&4\\
4&0&0&0\\
0&4&0&0\\
0&0&4&0
\end{array}\right).
\ee
The eigenvalues of $A_2$ are $\pm 4, \pm 4 \sqrt{-1}.$ Thus in this case we have $$\delta_0=4=\mathrm{dim} \ OG(2)+1.$$


\end{document}